\documentclass{article}%
\usepackage{amssymb}
\usepackage{amsfonts}
\usepackage{amsmath}
\usepackage{graphicx}%
\providecommand{\U}[1]{\protect\rule{.1in}{.1in}}
\newtheorem{theorem}{Theorem}

\newtheorem{corollary}{Corollary}

\newtheorem{example}{Example}

\newtheorem{lemma}[theorem]{Lemma}

\newtheorem{proposition}{Proposition}
\newtheorem{remark}{Remark}

\newenvironment{proof}[1][Proof]{\noindent\textbf{#1.} }{\ \rule{0.5em}{0.5em}}
\begin{document}

\title{Majorization bounds for distribution function}
\author{Ismihan Bairamov\\Department of Mathematics, Izmir University of Economics \\Izmir, Turkey. \ E-mail: ismihan.bayramoglu@ieu.edu.tr }
\maketitle

\begin{abstract}
Let $X$ be a random variable with distribution function $F,$ and $X_{1}%
,X_{2},...,X_{n}$ are independent copies of $X.$ Consider the order statistics
$X_{i:n},$ $i=1,2,...,n$ \ and denote $F_{i:n}(x)=P\{X_{i:n}\leq x\}.$ \ Using
majorization theory we write upper and lower bounds for $F$ expressed in terms
of mixtures of distribution functions of order statistics, i.e. $%
{\displaystyle\sum\limits_{i=1}^{n}}
p_{i}F_{i:n}$ and $%
{\displaystyle\sum\limits_{i=1}^{n}}
p_{i}F_{n-i+1:n}.$ \ It is shown that these bounds converge to $F$ \ for a
particular sequence $(p_{1}(m),p_{2}(m),...,p_{n}(m)),m=1,2,..$ as
$m\rightarrow\infty.$

\end{abstract}

\section{Introduction}

Let $X_{1},X_{2},...,X_{n}$ be independent and identically distributed (iid)
random variables with distribution function (cdf) $F$ and $X_{1:n}\leq
X_{2:n}\leq\cdots\leq X_{n:n}$ be the corresponding order statistics. Order
statistics are very important in the theory of statistics and its
applications. The theory of order statistics is well documented in David
(1981), David and Nagaraja (2003), Arnold et al. (1992) and in many research
papers dealing with different theoretical properties and applications of
ordinary order statistics and general models of ordered random variables.
Order statistics play a special role in statistical theory of reliability,
since they can be interpreted as the failure times of $n$ units with lifetimes
$X_{1},X_{2},...,X_{n}$ placed on a life test. A system of $n$ components is
called a $k$-out-of-$n$ system if it functions if and only if at least $k$
components function and therefore, life time of such a system is $X_{n-k+1:n}%
$. \ \ (see Barlow and Proschan, 1975). Consider a coherent system composed of
$n$ identical components with lifetimes $X_{1},X_{2},...,X_{n}$ having
distribution function $F$. Then the distribution function of the system
lifetime $T$ can be expressed as a convex combination of order statistics
$X_{1:n}\leq X_{2:n}\leq\cdots\leq X_{n:n}$ using Samanige signatures
$(p_{1},p_{2},...,p_{n})$ as follows:%
\begin{equation}
P\{T\leq x\}=%
{\displaystyle\sum\limits_{i=1}^{n}}
z_{i}F_{i:n}(x), \label{a1}%
\end{equation}
where $F_{i:n}(x)=P\{X_{i:n}\leq x\}$ and $z_{i}=P\{T=X_{i:n}\},$
$i=1,2,...,n$ are signatures (see Samaniego 2007). \ The system reliability
can be expressed as
\[
P\{T>x\}=%
{\displaystyle\sum\limits_{i=1}^{n}}
z_{i}\bar{F}_{i:n}(x),
\]
where $\bar{F}_{i:n}(x)=1-F_{i.n}(x).$ If $z_{i}=1/n,$ $i=1,2,...,n,$ then
$P\{T\leq x\}=\frac{1}{n}%
{\displaystyle\sum\limits_{i=1}^{n}}
F_{i:n}(x)=F(x).$ This means that if the system signature vector is $(\frac
{1}{n},\frac{1}{n},...,\frac{1}{n}),$ then the system reliability is the same
with the reliability of a single component. \ For a general coherent system,
assuming that the system reliability is known, is it possible to determine the
reliability of the components? The results presented in this paper allow to
answer partially this question, i.e. it follows that for a particular choice
of signatures, $\bar{F}$ can be approximated by the reliability of the system.

In general, in this note we consider mixtures of distribution functions of
order statistics $K_{n}(x):=%
{\displaystyle\sum\limits_{i=1}^{n}}
p_{i}F_{i:n}(x)$ and $H_{n}(x):=%
{\displaystyle\sum\limits_{i=1}^{n}}
p_{i}F_{n-i+1:n}(x)$ and using well known inequalities of majorization theory
we show that for a particular choice of $p_{i}$'s, $\ H_{n}(x)\leq F(x)\leq
K_{n}(x)$ for all $x\in%
\mathbb{R}
.$ \ It is shown that the similar inequalities can be written for the sample
mean and mixtures of order statistics. For a particular choice of vector
$\ (p_{1},p_{2},...,p_{n})$ the $L_{2}$ distance between $H_{n}(x)$ and
$K_{n}(x)$ can be made as small as we want. $\ $

\section{Main Results}

Let $\ \mathbf{a=(}a_{1},a_{2},...,a_{n})\in%
\mathbb{R}
^{n}$ , $\mathbf{b}=(b_{1},b_{2},...,b_{n})\in%
\mathbb{R}
^{n}$ and $a_{[1]}\geq a_{[2]}\geq\cdots\geq a_{[n]}$ denote the components of
$\mathbf{a}$ in decreasing order. The vector $\mathbf{a}$ is said to be
majorized by the vector $\mathbf{b}$ and $\text{denoted by }\mathbf{a}%
\prec\mathbf{b},$ if
\[
\sum_{i=1}^{k}a_{[i]}\leq\sum_{i=1}^{k}b_{[i]}\text{ for }k=1,2,\cdots,n-1
\]
and%
\[
\sum_{i=1}^{n}a_{[i]}=\sum_{i=1}^{n}b_{[i].}%
\]
The details of the theory of majorization can be found in Marshall et al.
(2011). The following two theorems are important for our study.

\begin{proposition}
\label{Proposition 1} Denote $D=\{(x_{1},x_{2},...,x_{n}):x_{1}\geq x_{2}%
\geq\cdots\geq x_{n}\},$ $\mathbf{a=(}a_{1},a_{2},...,a_{n}),$ $\mathbf{b}$
$=(b_{1},b_{2},...,b_{n}).$ The inequality
\[
\sum_{i=1}^{n}a_{i}x_{i}\leq\sum_{i=1}^{n}b_{i}x_{i}%
\]
holds for all $(x_{1},x_{2},...,x_{n})\in D$ if and only if $\mathbf{a}%
\prec\mathbf{b}$ in $D.$(Marshal et al. 2011, page 160).
\end{proposition}

\begin{proposition}
\bigskip\label{Proposition 2}The inequality
\[
\sum_{i=1}^{n}a_{i}x_{i}\leq\sum_{i=1}^{n}b_{i}x_{i}%
\]
holds whenever $x_{1}\leq x_{2}\leq\cdots\leq x_{n}$ if and only if
\begin{align*}%
{\displaystyle\sum\limits_{i=1}^{k}}
a_{i}  &  \geq%
{\displaystyle\sum\limits_{i=1}^{k}}
b_{i},\text{ }k=1,2,...,n-1\\%
{\displaystyle\sum\limits_{i=1}^{n}}
a_{i}  &  =%
{\displaystyle\sum\limits_{i=1}^{n}}
b_{i}.
\end{align*}

\end{proposition}

(Marshall et al. 2011, page 639).

\ \ 

Now, let $X_{1},X_{2},...,X_{n}$ be iid random variables with cdf $F,$ and
survival function $\bar{F}=1-F.$ Let $X_{1:n}\leq X_{2:n}\leq\cdots\leq
X_{n:n}$ be corresponding order statistics and $F_{i:n}(x)=P\{X_{i:n}\leq
x\}.$ We are interested in mixtures $%
{\displaystyle\sum\limits_{i=1}^{n}}
p_{i}\bar{F}_{i:n}(x)$ of cdf's of order statistics, where $\ p_{i}\geq0,$
$p_{1}\geq p_{2}\geq\cdots\geq p_{n}$ and $%
{\displaystyle\sum\limits_{i=1}^{n}}
p_{i}=1.$

Denote
\[
D_{+}^{1}=\{(x_{1},x_{2},...,x_{n}):x_{i}\geq0,i=1,2,...,n;\text{ }x_{1}\geq
x_{2}\geq\cdots\geq x_{n},\sum_{i=1}^{n}x_{i}=1\}.
\]

\begin{lemma}
Let $(p_{1},p_{2},...,p_{n})\in D_{+}^{1}$ . Then
\begin{equation}
H_{n}(x)\equiv%
{\displaystyle\sum\limits_{i=1}^{n}}
p_{i}F_{n-i+1:n}(x)\leq F(x)\leq%
{\displaystyle\sum\limits_{i=1}^{n}}
p_{i}F_{i:n}(x)\equiv K_{n}(x)\text{ for all }x\in%
\mathbb{R}
\label{a2}%
\end{equation}
and the equality holds if and only if $(p_{1},p_{2},\cdots,p_{n})=(\frac{1}%
{n},\frac{1}{n},...,\frac{1}{n}).$
\end{lemma}

\begin{proof}
Since $F_{1:n}(x)$ $\geq F_{2:n}(x)\geq\cdots\geq F_{n:n}(x)$ for all $x\in%
\mathbb{R}
,$ and $(\frac{1}{n},\frac{1}{n},...,\frac{1}{n})\prec(p_{1},p_{2}%
,...,p_{n}),$ the right hand side of the inequality (\ref{a2}) follows from
the Proposition 1 and left hand side follows from Proposition 2.
\end{proof}

\begin{corollary}
Let \ $\mathbf{p}=(p_{1},p_{2},...,p_{n})\in D_{+}^{1},$ $\mathbf{q=}%
(q_{1},q_{2},...,q_{n})\in D_{+}^{1}$ and $\mathbf{p}\prec\mathbf{q.}$ Then
\[%
{\displaystyle\sum\limits_{i=1}^{n}}
q_{i}F_{n-i+1:n}(x)\leq%
{\displaystyle\sum\limits_{i=1}^{n}}
p_{i}F_{n-i+1:n}(x)\leq F(x)\leq%
{\displaystyle\sum\limits_{i=1}^{n}}
p_{i}F_{i:n}(x)\leq%
{\displaystyle\sum\limits_{i=1}^{n}}
q_{i}F_{i:n}(x)
\]

\end{corollary}

\begin{example}
\bigskip\label{Example 1} Let $F(x)=x,$ $0\leq x\leq1.$ Then $H_{n}(x)=%
{\displaystyle\sum\limits_{i=1}^{n}}
q_{i}%
{\displaystyle\sum\limits_{k=i}^{n}}
\binom{n}{k}x^{k}(1-x)^{n-k}$ and $K_{n}(x)=%
{\displaystyle\sum\limits_{i=1}^{n}}
q_{i}%
{\displaystyle\sum\limits_{k=n-i+1}^{n}}
\binom{n}{k}x^{k}(1-x)^{n-k}.$ Let $n=3$ and $\mathbf{q}=(q_{1},q_{2}%
,q_{3})=(\frac{5}{9},\frac{3}{9},\frac{1}{9}).$ By simple calculations we
have
\[
H_{3}(x)=\frac{1}{3}(2x^{2}+x)\leq F(x)\leq\frac{1}{3}(5x-2x^{2})=K_{3}(x).
\]
Let $\mathbf{p}=(p_{1},p_{2},p_{3})=(\frac{6}{15},\frac{5}{15},\frac{4}{15}),$
then the functions $H_{3}(x)$ and $K_{3}(x)$ for this $\mathbf{p}$ are as
follows:
\[
H_{3}(x)=\frac{1}{5}(x^{2}+x)\leq F(x)\leq\frac{6}{5}(x-x^{2})=K_{3}(x).
\]
It is clear that $\mathbf{p}\prec\mathbf{q}.$ \ Below we present the graphs of
the functions $H_{3}(x),$ $F(x)=x,$ and $K_{3}(x)$ for two different vectors
$\mathbf{q}$ and $\mathbf{p:}$
\[
\]%
\begin{align*}
Figure\text{ 1. \ }H_{3}(x),\text{ }F(x) &  =x,\text{ and }K_{3}(x),\\
\text{for }(q_{1},q_{2},q_{3}) &  =(\frac{5}{9},\frac{3}{9},\frac{1}{9})\text{
and }(p_{1},p_{2},p_{3})=(\frac{6}{15},\frac{5}{15},\frac{4}{15}).
\end{align*}

\end{example}

For $(q_{1},q_{2},q_{3})=(\frac{5}{9},\frac{3}{9},\frac{1}{9}),$the $L_{2}$
distance between the functions $H_{3}(x),K_{3}(x)$ can be calculated and it is
$d(H_{3}(x),K_{3}(x))=%
{\displaystyle\int\limits_{0}^{1}}
(H_{3}(x)-K_{3}(x))^{2}dx=\frac{8}{135}\simeq0.059259.$ For $(p_{1}%
,p_{2},p_{3})=(\frac{6}{15},\frac{5}{15},\frac{4}{15})$ the distance is $%
{\displaystyle\int\limits_{0}^{1}}
(H_{3}(x)-K_{3}(x))^{2}dx=\frac{2}{375}\simeq0.005333.$

Note that the vector $(\frac{1}{n},\frac{1}{n},...,\frac{1}{n})$ is the
"smallest" in the sense of majorization, among the vectors $(p_{1}%
,p_{2},\cdots,p_{n})\in D_{+}^{1}$ , i.e. $(\frac{1}{n},\frac{1}{n}%
,...,\frac{1}{n})\prec(p_{1},p_{2},\cdots,p_{n})$ for all $(p_{1},p_{2}%
,\cdots,p_{n})\in D_{+}^{1}.$ $\ $Now the problem of interest is: \textit{for
a given }$n,$\textit{\ how small can the distance between }$H_{n}%
(x)$\textit{\ and }$K_{n}(x)$\textit{\ be made by appropriate choice of the
vector }$(p_{1},p_{2},\cdots,p_{n})?$ \ 

The following theorem answers this question.

\begin{theorem}
\label{Theorem 1A} There exists a sequence $\mathbf{p}(m)=(p_{1}%
(m),p_{2}(m),...,p_{n}(m))\in D_{+}^{1},$ $m=1,2,...$ such that%
\begin{equation}
H_{n}^{(m)}(x)\equiv%
{\displaystyle\sum\limits_{i=1}^{n}}
p_{i}(m)F_{n-i+1:n}(x)\leq F(x)\leq%
{\displaystyle\sum\limits_{i=1}^{n}}
p_{i}(m)F_{i:n}(x)\equiv K_{n}^{(m)}(x)\text{ for all }x\in%
\mathbb{R}
\label{a3}%
\end{equation}
and%
\begin{equation}
\underset{m\rightarrow\infty}{\ \lim}%
{\displaystyle\sum\limits_{i=1}^{n}}
p_{i}(m)F_{n-i+1:n}(x)=\underset{m\rightarrow\infty}{\ \lim}%
{\displaystyle\sum\limits_{i=1}^{n}}
p_{i}(m)F_{i:n}(x)=F(x)\text{ for all }x\in%
\mathbb{R}
. \label{a4}%
\end{equation}
Furthermore,
\begin{equation}%
{\displaystyle\int\limits_{-\infty}^{\infty}}
\left\vert K_{n}^{(m)}(x)-H_{n}^{(m)}(x)\right\vert dx=o(\frac{1}{m^{1-\alpha
}}),\text{ }0<\alpha<1. \label{a5}%
\end{equation}

\end{theorem}

\begin{proof}
\bigskip Consider $p_{i}(m)=\frac{m+n-i+1}{a_{n}(m)},$ $i=1,2,...,n;$
$\ \ m\in\{0,1,2,...\},$ where $a_{n}(m)=nm+\frac{n(n+1)}{2}.$ It is clear
that $p_{1}(m)\geq p_{2}(m)\geq\cdots\geq p_{n}(m)$ and $%
{\displaystyle\sum\limits_{i=1}^{n}}
p_{i}(m)=1.$ Since $(\frac{1}{n},\frac{1}{n},...,\frac{1}{n})\prec
(p_{1}(m),p_{2}(m),...,p_{n}(m))$ then from Lemma 1 we have
\begin{equation}%
{\displaystyle\sum\limits_{i=1}^{n}}
p_{i}(m)F_{n-i+1:n}(x)\leq F(x)\leq%
{\displaystyle\sum\limits_{i=1}^{n}}
p_{i}(m)F_{i:n}(x). \label{3a}%
\end{equation}
Since
\[
\underset{m\rightarrow\infty}{\lim}p_{i}(m)=\underset{m\rightarrow\infty}%
{\lim}\frac{m+i}{nm+\frac{n(n+1)}{2}}=\frac{1}{n},\text{ }i=1,2,...,n,
\]
and (\ref{a4}) follows. To prove (\ref{a5}) consider the $L_{1}$ distance
between $K_{n}^{(m)}(x)$ and $H_{n}^{(m)}(x).$ We have
\begin{align*}
\Delta_{m}  &  \equiv%
{\displaystyle\int\limits_{-\infty}^{\infty}}
\left\vert K_{n}^{(m)}(x)-H_{n}^{(m)}(x)\right\vert dx\\
&  =%
{\displaystyle\int\limits_{-\infty}^{\infty}}
\left\vert
{\displaystyle\sum\limits_{i=1}^{n}}
p_{i}(m)F_{i:n}(x)-%
{\displaystyle\sum\limits_{i=1}^{n}}
p_{i}(m)F_{n-i+1:n}(x)\right\vert dx\\
&  =%
{\displaystyle\int\limits_{-\infty}^{\infty}}
\left\vert
{\displaystyle\sum\limits_{i=1}^{n}}
p_{i}(m)F_{i:n}(x)-F(x)+F(x)-%
{\displaystyle\sum\limits_{i=1}^{n}}
p_{i}(m)F_{n-i+1:n}(x)\right\vert dx\\
&  =%
{\displaystyle\int\limits_{-\infty}^{\infty}}
\left\vert
{\displaystyle\sum\limits_{i=1}^{n}}
(p_{i}(m)-\frac{1}{n})F_{i:n}+(\frac{1}{n}-p_{i}(m))F_{n-i+1:n}(x)\right\vert
dx\\
&  \leq%
{\displaystyle\sum\limits_{i=1}^{n}}
\left\vert p_{i}(m)-\frac{1}{n}\right\vert
{\displaystyle\int\limits_{-\infty}^{\infty}}
\left\vert F_{i:n}(x)-F_{n-i+1:n}(x)\right\vert dx\\
&  \leq(p_{1}(m)-\frac{1}{n})c_{n}=\frac{\frac{1}{m}\frac{n(n+1)}{2}}%
{n^{2}+\frac{n^{2}(n+1)}{2}\frac{1}{m}}c_{n},
\end{align*}
where $c_{n}=%
{\displaystyle\sum\limits_{i=1}^{n}}
{\displaystyle\int\limits_{-\infty}^{\infty}}
\left\vert F_{i:n}(x)-F_{n-i+1:n}(x)\right\vert dx.$
\end{proof}

In Figure 3 the graphs of $H_{3}(x),$ $K_{3}(x)$ for $n=3$ in case of standard
normal distribution $N_{0,1}(x)$ for a vector ($p_{1},p_{2},p_{3}%
)=(2/3,2/9,1/9),$ which clearly is not a member of the sequence $\mathbf{p}%
(m).$The numerical calculations in Maple 13 show that $%
{\displaystyle\int\limits_{-\infty}^{\infty}}
\left\vert K_{n}^{(m)}(x)-H_{n}^{(m)}(x)\right\vert dx=0.30903.$%

\[
\]%
\begin{align*}
Figure\text{ 2. Graphs of }H_{3}(x) &  \leq N_{0,1}(x)\leq K_{3}(x),\text{ }\\
n &  =3\text{ and (}p_{1},p_{2},p_{3})=(2/3,2/9,1/9)\
\end{align*}

The members of the sequence $\mathbf{p}(m),m=1,2,...$ are most "uniform", and
according to the basic idea of majorization they must allow better
approximation than any other vector. To illustrate the rate of convergence in
case of standard normal distribution we present in Figure 3 below, \ the
graphs of $K_{n}^{(m)}(x)$ $\leq N_{0,1}(x)\leq H_{n}^{(m)}(x).$ Some
numerical values of $\Delta_{m}$ for different values of $m$ are presented in
Table 1.%
\begin{align*}
& \\
&
\end{align*}

\bigskip\ \ \ \ \ \ \ \ \ \ %

\begin{align*}
Figure\text{ }3.\text{ Graphs of }K_{n}^{(m)}(x) &  \leq N_{0,1}(x)\leq
H_{n}^{(m)}(x)\text{ ,}n=3\\
m &  =2,3,10
\end{align*}

\ \ \ %

\[
Table\text{ }1.\text{Values of }\Delta_{m}%
\]

\begin{align*}
&
\begin{tabular}
[c]{|l|l|l|l|l|l|}\hline
$m$ & 1 & 2 & 3 & 4 & 5\\\hline
$\Delta_{m}$ & 0.34337 & 0.13735 & 0.10301 & 0.08241 & 0.06867\\\hline
\end{tabular}
\\
&
\begin{tabular}
[c]{|l|}\hline
$m$\\\hline
$\Delta_{m}$\\\hline
\end{tabular}%
\begin{tabular}
[c]{|l|l|l|l|l|}\hline
10 & 15 & 20 & 25 & 30\\\hline
0.03434 & 0.02423 & 0.018729 & 0.01526 & 0.01288\\\hline
\end{tabular}
\end{align*}

\begin{remark}
\label{Remark 1}It is clear that using Proposition 1 and 2 and using order
statistics $X_{i.n},$ instead of $F_{i:n}$ we have similar to Lemma 1 and
Theorem 1 results for order statistics. Let $\bar{X}=%
{\displaystyle\sum\limits_{i=1}^{n}}
X_{i}.$ Then, for a sequence $\mathbf{p}(m)=(p_{1}(m),p_{2}(m),...,p_{n}%
(m))\in D_{+}^{1},$ $m=1,2,...$ \ it is true that
\begin{equation}
X_{m}^{L}\equiv%
{\displaystyle\sum\limits_{i=1}^{n}}
p_{i}(m)X_{i:n}\leq\bar{X}\leq%
{\displaystyle\sum\limits_{i=1}^{n}}
p_{i}(m)X_{n-i+1:n}=X_{m}^{U}\text{ a.s.} \label{4a}%
\end{equation}
and%
\begin{equation}
\underset{m\rightarrow\infty}{\ \lim}\left[
{\displaystyle\sum\limits_{i=1}^{n}}
p_{i}(m)X_{n-i+1:n}-%
{\displaystyle\sum\limits_{i=1}^{n}}
p_{i}(m)X_{i:n}\right]  =0\text{ \ a.s..} \label{5a}%
\end{equation}
Furthermore,
\[
E\left\vert X_{m}^{U}-X_{m}^{L}\right\vert =o(\frac{1}{m^{1+\alpha}}),\text{
}0<\alpha<1.
\]
From (\ref{5a}) it follows that
\begin{equation}%
{\displaystyle\sum\limits_{i=1}^{n}}
p_{i}(m)\mu_{i:n}\leq E(X)\leq%
{\displaystyle\sum\limits_{i=1}^{n}}
p_{i}(m)\mu_{n-i+1:n}, \label{6a}%
\end{equation}
where $\mu_{i:n}=E(X_{i:n}).$ \ The rate of convergence can be estimated as
\begin{align*}
\delta_{m}  &  =\left\vert
{\displaystyle\sum\limits_{i=1}^{n}}
p_{i}(m)\mu_{n-i+1:n}-%
{\displaystyle\sum\limits_{i=1}^{n}}
p_{i}(m)\mu_{i:n}\right\vert \\
&  \leq\frac{\frac{1}{m}\frac{n(n+1)}{2}}{n^{2}+\frac{n^{2}(n+1)}{2}\frac
{1}{m}}C_{n},
\end{align*}
where $C_{n}=%
{\displaystyle\sum\limits_{i=1}^{n}}
\left\vert \mu_{n-i+1:n}-\mu_{i:n}\right\vert .$
\end{remark}

\end{document}